\documentclass{amsart}

\usepackage[latin1]{inputenc}
\usepackage[english]{babel}
\usepackage{amssymb}
\usepackage{mathrsfs}
\usepackage{graphicx}
\usepackage[all]{xy}
\usepackage{color}
\usepackage[pagebackref]{hyperref}

\newcommand{\CC}{\mbox{${\mathbb C}$}}
\newcommand{\F}{\mbox{${\mathscr F}$}}

\newcommand{\sA}{\mbox{${\mathscr A}$}}

\renewcommand{\O}{{\mathcal O}}

\renewcommand{\dim}{\mathrm{dim}\,}

\def\C{\mathbb C}

\def\T{T}

\newtheorem{lema}{Lemma}[section]
\newtheorem{cor}[lema]{Corollary}
\newtheorem{teo}[lema]{Theorem}
\newtheorem*{teorema}{Theorem}

\newtheorem{prop}[lema]{Proposition}
\theoremstyle{definition}

\newtheorem{remark}[lema]{Remark}
\newtheorem{defi}[lema]{Definition}

\newtheorem{exe}[lema]{Example}

\begin{document}

\title[Families and unfoldings of singular   holomorphic Lie Algebroids]
{Families and unfoldings of singular   holomorphic Lie Algebroids}

\author[M. Corr\^ea]{Mauricio Corr\^ea\textsuperscript{1}}
\thanks{\textsuperscript{1} The author was fully supported by the CNPQ grants number 202374/2018-1, 302075/2015-1, and 400821/2016-8.}
\dedicatory{}
\address{}
\email{ }
\dedicatory{}
\address{}
\email{ }
\author[  A. Molinuevo]{ Ariel Molinuevo\textsuperscript{2}}
\thanks{\textsuperscript{2} The author was fully supported by Universidade Federal do Rio de Janeiro, Brazil.}
\dedicatory{}
\address{}
\email{}

\author[  F. Quallbrunn]{ Federico Quallbrunn\textsuperscript{3}}
\thanks{\textsuperscript{3} The author was fully supported by CONICET, Argentina.}
\dedicatory{}
\address{}
\email{ }
\date{}

\thanks{ }
\date{ }


\begin{abstract}
In this paper, we investigate families of singular holomorphic Lie algebroids on complex analytic spaces. We introduce and study a special type of deformation called \emph{unfoldings of Lie algebroids}, which generalizes the  theory of singular holomorphic foliations developed by T. Suwa.
We show that  a one to one correspondence between transversal unfoldings and holomorphic flat connections on a natural Lie algebroid on the bases exists.  
\end{abstract}

\maketitle

\section{Introduction}

At the beginning of the 1980's T. Suwa introduced the notion of \emph{unfolding of a codimension one foliation} in \cite{Suwa} in order to study the singularities of complex analytic foliations. He continued developing the theory in the articles \cite{suwa-multiform},\cite{suwa-complex},\cite{suwameromorphic} and others. His studies were about a certain type of infinitesimal perturbation of a codimension one foliation $\F_0=(\omega)$ defined by an integrable differential form $\omega\in \Omega^1_{\CC^n}$. This particular type of perturbations, named \emph{unfoldings}, define a codimension one foliation in the bigger space $\CC^n \times S$, where $S$ is the analytic spectrum of $\CC[\varepsilon]/(\varepsilon^2)$ such that it restricts to the original foliation in $\CC^n\times\{0\}$.
Then an unfolding of $\F_0=(\omega)$ is a foliation $\F=(\omega_{\varepsilon})$ where $\omega_\varepsilon$ is  a differential form in $\Omega^1_{\CC^{n} \times S|\CC}$. One can write  $\omega_\varepsilon = \omega + \varepsilon \eta + h d\varepsilon$ plus the condition of integrability on $\omega_\varepsilon$.

This contrasts with the definition of \emph{deformation} of a codimension one foliation $\omega$, because a deformation, parametrized by $S$, defines a relative differential form over $\Omega^1_{\CC^n\times S|S}$, such that it restricts to the original foliation in $\CC^n\times\{0\}$. Given this, one can write such deformation as $\widetilde{\omega}_\varepsilon=\omega + \varepsilon \eta$ plus the condition of integrability on $\widetilde{\omega}_\varepsilon$.

Remember that the integrability condition on $\omega$ is given by the equation
\[
 \omega\wedge d\omega = 0\ .
\]


\

In a more general setting we can change the space $\CC^n$ to $X$, a non-singular algebraic variety over an algebraically closed field  $k$. Similarly, we can change the infinitesimal parameter space $(\CC,0)$ to $S$, meaning any smooth scheme variety of finite type over the base field of $X$. Finally, we consider \emph{global} foliations and unfoldings. Now, let us be more precise and consider the following definitions:

\begin{defi}(Foliation)
 Let $X$ be a separated scheme of finite type over an alebraically closed field $k$. An \emph{involutive distribution} over $X$ is a subsheaf $T\F_0\subset TX$ that is closed under the Lie brackets. This way for every open set $U\subset X$ we have
\[
 [T\F_0(U),T\F_0(U)]\subset T\F_0(U)\ .
\]
The annihilator of an involutive distribution defines an \emph{integrable Pfaff system}. We define $\F_0$ to be a \emph{foliation}, if $\F_0$ is given by an involutive distribution or by an integrable Pfaff system. 
 \end{defi}

  \begin{remark}
  Note that with the above definition we are allowing a foliation to have singularities.
 \end{remark}


\begin{defi}(Unfolding of a foliation) Let $\F_0$ be a foliation in $X$ and $s\in S$ be a closed point. 
An \emph{unfolding} of $\F_0$ parametrized by $(S,s)$ is a foliation $\F$ on $X\times S$ such that:
\begin{enumerate}
 \item[i)] The restriction of $\F$ to $X\times \{s\}$ is $\F_0$.
 \item[ii)] The dimensions of $\F$ and $\F_0$ verify the formula $dim(\F) = dim(\F_0) + dim(S)$.
\end{enumerate}
 We say the unfolding $\F$ of $\F_0$ is \emph{isotrivial}, if it induces a trivial (\emph{i.e.}: constant) family of Pfaff systems equivalently of distributions.
\end{defi}

\begin{defi}(Transversal unfolding)
 Given a foliation $\F_0$ and an unfolding $\F$ of $\F_0$ parametrized by $(S,s)$ we have exact sequences
\[
 \xymatrix{
 0 \ar[r]^{} &T_S(\F) \ar[r] \ar@{^(->}[d] &T_S(X\times S) \ar[r]\ar@{^(->}[d]&N_S(\F)\ar[r] \ar@{^(->}[d]& 0 \\
 0 \ar[r]& T\F\ar[r]& T(X\times S)\ar[r]&N\F \ar[r]& 0\ .
 }
\]
We say that $\F$ is \emph{transversal}, if $N\F/N_S\F = 0$. And $T_S(\F)$ denotes the tangent space of the foliation $\F$, see \cite[Remark 2.5]{Quallbrunn} for a definition.
 \end{defi}

In  \cite[Theorem 2.12]{Quallbrunn} F. Quallbrunn showed the following result:
\begin{teorema}\label{correspondencia3}
Let $X$ be a non-singular variety and $\F_0$ a foliation on $X$.
 There is, for each scheme $S$, a $1$ to $1$ correspondence:
\[
\left\{\begin{aligned}
\text{isotrivial transversal unfoldings } \\
\text{of $\F$ parametrized by $S$} 
\end{aligned} \right\}\longleftrightarrow \left\{
\begin{aligned}
\text{sections }
\upsilon \in H^0(S,\Omega^1_S)\otimes H^0(X,\mathfrak{t}(\F_0))\\
 \text{ s.t.: } d\upsilon+\frac{1}{2} [\upsilon,\upsilon]=0
\end{aligned}
\right\}
\]
where $\mathfrak{t}(\F_0)$ is the subsheaf of $N_S\F_0$ defined by
\[
 \mathfrak{t}(\F_0)(U) = \{x\in N_S\F_0(U)\ \text{s.t.:}\ [T_S\F_0,x]=0\}
\]
for an open set $U\subset X\times S$.

Notice that the since the foliation is closed under brackets then the bracket operation of vector fields defines an action of $T_S\F$ in $N_S\F$.
\end{teorema}

\

There is a similar result in the real case  in \cite[Theorem 4.6]{Sheng}.

\

It is well known that every foliation is induced by a Lie algebroid, see Definition \ref{lie_algebroid} below for a proper definition. 

 We observe that  a section $ \upsilon \in H^0(S,\Omega^1_S)\otimes H^0(X,\mathfrak{t}(\F_0))$, satisfying   $d\upsilon+\frac{1}{2} [\upsilon,\upsilon]=0$,
 is such, that the induced $\O_S$-morphism
 \[
 \pi_* \upsilon:  \pi_* \mathfrak{t}(\F_0) \to  T_S\ ,
 \]
where $\pi:X\times S \to S$ is the projection, is the anchor map of the  Lie algebroid associated with the foliation
\[
 \pi_* \upsilon( \pi_* \mathfrak{t}(\F_0))\subset T_S\ .
\]

In this work we will generalize all the given definitions to the context of Lie algebroids and we will prove a similar result of the above theorem in that situation.
In this way we propose to study features of Lie algebroids that may be regarded as generalizations of singularities of foliations.
The applications we have in mind include the study of Poisson structures and sheaves of logarithmic forms. 
Given a family of algebroids $\sA_S$, we define a sheaf $\mathfrak{u}(\sA_S)$ with a Lie bracket. In Theorem \ref{correspondencia2} we give a correspondence between transversal unfoldings of a family of algebroids $\sA_S$ and morphisms
\[
 \pi^{-1}(T_S) \to \mathfrak{u}(\sA_S)
\]
respecting the brackets. As we prove in Proposition \ref{propLieAlgebroid}, the sheaf $\pi_*\mathfrak{u}(\sA_S)$ is a Lie algebroid and so the maps of Theorem \ref{correspondencia2} are in bijection with flat connections on the algebroid $\pi_*\mathfrak{u}(\sA_S)$. Flat connections on an algebroid are in correspondence with forms that have values in that same algebroid, satisfying Maurer-Cartan equations. In this way we arrive at a formulation which is similar but more general to that of \cite{Quallbrunn},

\[
\left\{
\text{ transversal unfoldings of $\sA_S$} 
\right\}\longleftrightarrow \left\{
\begin{aligned}
\text{sections } \upsilon \in H^0(S,\Omega^1_S\otimes \pi_*\mathfrak{u}(\sA_S))
\\ \text{ s.t.: } d\upsilon+\frac{1}{2} [\upsilon,\upsilon]=0
\end{aligned}
\right\}.
\]

In order to arrive to this conclusion, we investigate the structure as a Lie algebroid of the sheaf $D^{\leq 1}_X(\F)$ of differential operators of order less or equal to $1$ of a sheaf $\F$. 
The principal result here is Theorem \ref{propPullBackD}. 

\

Finally, in Section \ref{sec4} we illustrate and comment our results with several examples of unfoldings in different contexts.

\section{Lie Algebroids}

\begin{defi}(Lie algebroid)\label{lie_algebroid}
Let $\sA$ be a reflexive sheaf of $\O_X$-modules over a complex
manifold~$X$, equipped with a $\O_X$-morphism $a\colon\sA\to\T_X$.
We say that $\sA$ is a \emph{Lie algebroid} of \emph{anchor~$a$} if
there is a $\C$-bilinear map
$\{\cdot\,,\cdot\}\colon\sA\otimes_{\O_X}\sA\to\sA$  such that
\begin{itemize}
\item[(a)] $\{v,u\}=-\{u,v\}$;
\item[(b)] $\{u,\{v,w\}\}+\{v,\{w,u\}\}+\{w,\{u,v\}\}=0$;
\item[(c)] $\{g\cdot u,v\}=g\cdot\{u,v\}-a(v)(g)\cdot u$ for all $g\in\O_X$ and
$u$,~$v\in\sA$.
\end{itemize}
\end{defi}

The singular set of $\sA$ is defined by
$$
\mathrm{Sing}(\sA) = \mathrm{Sing}(\mathrm{Coker}(a)).
$$
The Lie algebroid $a\colon\sA\to\T_X$ induces a holomorphic foliation $Im(a)\subset T_X$. 

\begin{defi}(Pullback of a Lie algebroid) Given a Lie algebroid $\sA$ over a variety $X$ and a morphism $f:Y\to X$ we define the \emph{pullback algebroid} $f^\bullet \sA$ over $Y$. The underlying sheaf of $f^\bullet\sA$ is defined as the fibered product of the diagram
\[
 \xymatrix{
 f^\bullet\sA:=f^*\sA \oplus_{f^*T_X} T_Y \ar[r] \ar[d] & T_Y \ar[d]^{Df}\\
 f^*\sA\ar[r]^{a} & f^*T_X 
 }
\]
The anchor map is the top horizontal map in the above diagram. 
The Lie algebra structure is induced by the restriction of the bracket in the direct sum $f^*\sA \oplus T_Y$.
 
\end{defi}

Let  $f:X\to  S$ be a smooth morphism of analytic
spaces  and consider $T_{X|S}$ the relative tangent sheaf, which is naturally a subsheaf of $T_X$.

\begin{defi}
A \emph{family} of  singular holomorphic   Lie algebroids over $X$ is a  reflexive sheaf $\sA$ of modules over $X$ which is flat over $S$, equipped with a $\O_X$-morphism
 $a_S\colon\sA\to\T_{X|S}$  and an $f^{-1}\O_S$-linear map
$\{\cdot\,,\cdot\}_S\colon\sA\otimes_{f^{-1}\O_S}\sA\to\sA$ 
such that
\begin{itemize}
\item[(a)] $\{\alpha,\beta\}=-\{\beta,\alpha\}$;
\item[(b)] $\{\alpha,\{\beta,\gamma\}\}+\{\beta,\{\gamma,\alpha\}\}+\{\gamma,\{\alpha,\beta\}\}=0$;
\item[(c)] $\{f\cdot \alpha,\beta\}=f\{\alpha,\beta\}-a(\beta)(f)\cdot \alpha$ for all $f\in\O_X$ and
$\alpha$,~$\beta \in\sA$.
\end{itemize}
\end{defi}

\begin{remark}\label{remarkpullback} Given a family $\sA_S$ over a smooth morphism $p:X\to S$ and a morphism $f:R\to S$ we consider the fibered product
 \[
  \xymatrix{
    X_R \ar[r]^f \ar[d]_{p_R} & X \ar[d]^p\\
    R \ar[r]^f & S
  }
 \]
If we take the Lie algebroid pull-back $f^\bullet \sA$ we obtain a family over the smooth morphism $p_R:X_R\to R$. Observe that, using a covering of $X$ by open sets of the form $Y\times U$ with $U\subseteq S$ an open set, we can give local generators $\frac{\partial }{\partial y_1},\dots, \frac{\partial }{\partial y_d}$ of $T_{X|S}$ where $y_1,\dots, y_d$ are local coordinates of $Y$. Using such a covering we obtain a covering for $X_R$ of the form $Y\times V$ where $V=f^{-1}(U)\subseteq R$ which give local generators of $T_{X_R|R}$ of the form $\frac{\partial }{\partial y_1},\dots, \frac{\partial }{\partial y_d}$.

In other words we have $f^*T_{X|S}\simeq T_{X_R|R}$, so the underlying sheaf of the algebroid $f^\bullet \sA$ in this case is just $f^*\sA$.
\end{remark}

\begin{exe}
 An important caveat to have in mind here is that, contrary to what one may suppose, the algebroid pull-back is not an associative operation. In other words, in general $g^\bullet f^\bullet \sA \ncong (f\circ g)^\bullet \sA $. There is however a canonical morphism $c_{gf}:g^\bullet f^\bullet \sA \to (f\circ g)^\bullet \sA $ given by the pull-back property of $(f\circ g)^\bullet \sA$, but is not an isomorphism in general.  
 
 To see an example of this lets take $X=\mathbb{A}^2$, and $a:\sA\to T_\mathbb{A}^2$ to be the inclusion of the sheaf generated by the vector field $v=y\frac{\partial}{\partial x}+x\frac{\partial}{\partial y}$. Let $f:\mathbb{A}^1\to \mathbb{A}^2$ be the inclusion of the axis $(y=0)$ and $g:(0,0)\to \mathbb{A}^1$ the inclusion of the origin in the axis. It follows from the definition that $f^\bullet \sA$ is the pull-back in the diagram
 \[
  \xymatrix{
    (0) \ar[r] \ar[d] & T_{\mathbb{A}^1} \ar[d]\\
    f^*\sA \ar[r]_{[v] \mapsto [x\frac{\partial}{\partial y}] } & f^*T_{\mathbb{A}^2}
  }
 \] 
 so $f^\bullet \sA=(0)$, and therefore $g^\bullet f^\bullet \sA=(0)$. But the pull-back of the inclusion of the origin in the plane gives $(fg)^\bullet\sA$ which is 
  \[
  \xymatrix{
    (fg)^\bullet \sA \ar[r] \ar[d]_{\simeq} & (0) \ar[d]\\
    (fg)^*\sA \ar[r]_{[v] \mapsto 0 } & (fg)^*T_{\mathbb{A}^2}
  }
 \]
 so $(fg)^\bullet \sA $ is an algebroid over the point consisting of a vector space of dimension $1$ with trivial bracket and trivial anchor, \emph{i.e.}: an abelian one dimensional Lie algebra.  
\end{exe}

\section{Unfolding of Lie algebroids }

\begin{defi}(Unfolding of a Lie algebroid)
Let  $a_S:(\sA_S,\{\cdot\,,\cdot\}_S)\to T_{X|S}$ be a family of holomorphic Lie algebroids over a smooth morphism $\pi:X\to S$. 
An \emph{unfolding} of $\sA_S$  is a Lie algebroid $\sA$ on $ X$ with anchor $a:\sA\to T_X$ such 
that
\begin{itemize}
\item[(a)]  The family of Lie algebroids $\sA_S$, is recovered as $\sA_S= a^{-1}(Im(a_S))$. 
\item[(b)]  $\mathrm{rank}(\sA) =  \mathrm{rank}(\sA_S)  + \dim(S).$
\end{itemize}
An unfolding is defined to be \emph{isotrivial} if the induced flat family of algebroids is trivial. 
\end{defi}

\begin{defi}(Transversal unfolding)
 Let $\sA$ be an unfolding of a family $\sA_S$. Let $N\sA$ be the cokernel of the map $a:\sA\to T_X$ and $N_S\sA$ be the cokernel of the map $a_S:\sA_S\to T_{X|S}$. Notice that the maps $\sA_S\to \sA$ and $T_{X|S}\to T_{X}$ induce a map $N_S\sA\to N\sA$. The unfolding $\sA$ is said to be \emph{transversal} if the map $N_S\sA\to N\sA$ is an isomorphism.
\end{defi}

%
%
%
%
%

\begin{defi}(Differential operator)
 A \emph{differential operator} $\psi$ of order $\leq 1$ on a quasi-coherent sheaf $\sA$ over $X$ is a morphism of $\CC$-modules $\psi:\sA\to\sA$ such that for each local section $f\in \O_X(U)$ there is a local section $\psi f\in \O_X(U)$ such that, if $m_f:\sA|_U\to \sA|_U$ is multiplication by $f$, then 
 \[
\psi|_U\circ m_f - m_f\circ \psi|_U=m_{\psi f}.  
  \]
  We denote by $D^{\leq 1}_{X}(\sA)$ the $\O_X$-module of differential operators of order $\leq 1$ of $\sA$.
\end{defi}

\begin{defi}
 Given a differential operator $\psi$ of order $\leq 1$ the map $f\mapsto \psi f$ determines a derivation of $\O_X$. This derivation is called \emph{the symbol} of the operator $\psi$.
\end{defi}

\begin{remark}
 For any torsion free sheaf $\F$, the sheaf $D^{\leq 1}_X(\F)$ has a natural structure of Lie algebroid, as the commutator of differential operators define a bracket whose anchor is the symbol.
\end{remark}

\begin{remark}
The above definition of the sheaf of differential operator differs from that of \cite[16.8]{egaIV}. In \emph{loc. cit.} a differential operator between two sheaves $\F$ and $\mathscr{G}$ is an element of the sheaf $\mathcal{H}om_X(\mathcal{P}^1_X(\F),\mathscr{G})$, where $\mathcal{P}^1_X(\F)$ is the \emph{sheaf of principal parts of order $1$} of $\F$ (see \cite[16.7]{egaIV}). in the case $\F=\mathscr{G}=\O_X$ we have a short exact sequence
\[
 0\to \Omega^1_X\to \mathcal{P}^1_X\to \O_X\to 0,
\]
tensoring with a sheaf $\F$ gives the sequence $0\to \Omega^1_X\otimes \F\to \mathcal{P}^1_X(\F)\to \F\to 0$. Then applying $\mathcal{H}om_X(-,\mathscr{F})$ gives the exact sequence
\[
0\to \mathcal{E}nd(\F)\to \mathcal{H}om_X(\mathcal{P}^1_X(\F),\F) \to T_X\otimes  \mathcal{E}nd(\F)\to \mathcal{E}xt^1_X(\F,\F).
\]
Using the natural map $\O_X \to \mathcal{E}nd(\F)$ we have the following diagram
\[\xymatrix{
0 \ar[r] & \mathcal{E}nd(\F) \ar[r] \ar@{=}[d] & D^{\leq 1}_{X}(\F) \ar[r] \ar[d] & T_X \ar[d]  \\
0 \ar[r] & \mathcal{E}nd(\F) \ar[r] & \mathcal{H}om_X(\mathcal{P}^1_X(\F),\F) \ar[r] & T_X\otimes  \mathcal{E}nd(\F) 
}
\]
In which the right square is a pull-back diagram.
\end{remark}

\begin{lema}\label{lemafDD}
 Let $\F$ be a sheaf over $X$, and let $f:R\to S$ be a morphism. Denote by $Y$ the pull-back of $X$ by $f$ and by $f^*\F$ the corresponding pull-back of $\F$ as a sheaf over $Y$. There is a canonical morphism 
 \[
  f^\bullet D^{\leq 1}_X(\F)\to D^{\leq 1}_{Y}(f^*\F).
 \]
\end{lema}

\begin{proof}
  To a local section of $f^\bullet D^{\leq 1}_X(\F)$ we can explicitly assign a differential operator on $f^*\F$ by the following formula. A local section $\alpha$ of $f^\bullet D^{\leq 1}_X(\F)$ is by definition 
 \[
  \alpha = \left(\sum_i \psi_i \otimes r_i , v \right),
 \]
where $\psi_i\otimes r_i$ are local sections of $D^{\leq 1}_X(\F)\otimes \O_{X_R}$ and $v$ is a local section of $T_{Y}$ such that $\sigma\left(\psi_i\otimes r_i\right) = Df(v)$.
Now, given a local section $a=\sum_i a_i \otimes t_i$ of $\F\otimes\O_{Y}$ we define 
\[
 \alpha(a):= \sum_{ij} (\psi_i(a_j)\otimes r_it_j +a_j \otimes v(t_j)s_i \in f^*\F .
\]
\end{proof}

\begin{lema}\label{lemaPartesPrincipales}
 Let $\F$ be such that both  $\F$ and $f^*\F$ are \emph{reflexive} sheaves over $X$ and $Y$ respectively, then there is a morphism 
 \[
  \mathcal{H}om_Y(\mathcal{P}^1_Y(f^*\F),f^*\F)\to f^*\left(\mathcal{H}om_X(\mathcal{P}^1_X(\F),\F)\right) .
 \]
\end{lema}
\begin{proof}
 By \cite{egaIV} there is always a morphism 
 \[
  \mathcal{H}om_Y(\mathcal{P}^1_Y,\O_Y)\to f^*\left(\mathcal{H}om_X(\mathcal{P}^1_X,\O_X)\right)
 \]
 In the case where $\F$ is locally free we have canonical isomorphisms
 \[ 
 \mathcal{H}om_Y(\mathcal{P}^1_Y(f^*\F),f^*\F) \cong \mathcal{H}om_Y(\mathcal{P}^1_Y(f^*\F),\O_Y) \otimes f^*\F
 \] 
 and 
 \[
 f^*\left(\mathcal{H}om_X(\mathcal{P}^1_X,\O_X)\right)\otimes f^*\F \cong  
  f^*\left(\mathcal{H}om_X(\mathcal{P}^1_X(\F),\F)\right).
  \]
  Then from the above morphism between the sheaves of principal parts we get
 \[
   \mathcal{H}om_Y(\mathcal{P}^1_Y(f^*\F),\O_Y) \otimes f^*\F   
  \to 
  f^*\left(\mathcal{H}om_X(\mathcal{P}^1_X,\O_X)\right)\otimes f^*\F ,
 \]
 which gives us a morphism $\mathcal{H}om_Y(\mathcal{P}^1_Y(f^*\F),f^*\F)\to f^*\left(\mathcal{H}om_X(\mathcal{P}^1_X(\F),\F)\right) $ if $\F$ is locally free.
 If more generally $\F$ is reflexive then there is an open set $U$ such that $\F|_U$ is locally free and such that every section of $\F(U)$ extends to a global section, in other words, if $j:U\to X$ is the inclusion then we have an isomorphism $j_*(\F|_U)\simeq \F$. If also $f^*\F$ is reflexive then $f^*\F|_{f^{-1}U}$ is locally free and $j_*(f^*\F|_{f^{-1}U})\simeq f^*\F$.
 
 Local sections of $\mathcal{H}om_X(\mathcal{P}^1_X(\F),\F)$ over an open set $V$ are in natural correspondence with $\C$-linear morphisms $\F(V)\to\F(V)$ that are differential maps in the sense of \cite{egaIV}. In particular every section of $\mathcal{H}om_X(\mathcal{P}^1_X(\F),\F)(U)$ extends to a global section. And the same can be said about sections of $\mathcal{H}om_Y(\mathcal{P}^1_Y(f^*\F),f^*\F)$ over $(f^{-1}U)$ extending to global sections. Now over $f^{-1}U$ we have a morphism
 \[
  \mathcal{H}om_Y(\mathcal{P}^1_Y(f^*\F),f^*\F)|_{f^{-1}U}\to f^*\left(\mathcal{H}om_X(\mathcal{P}^1_X(\F),\F)|_U\right).
 \]
 Which extends to a morphism $\mathcal{H}om_Y(\mathcal{P}^1_Y(f^*\F),f^*\F) \to f^*\left(\mathcal{H}om_X(\mathcal{P}^1_X(\F),\F)\right)$ as wanted.
\end{proof}

\begin{lema}\label{lemaEnd}
 Let $\F$ be a reflexive sheaf over $X$ and $U\subseteq X$ be an open set such that $\F|_U$ is locally free and such that every section of $\O_X(U)$ extends to a global section. Then there is a section $\mathcal{E}nd(\F)\to\O_X$ to the canonical inclusion $\O_X\to \mathcal{E}nd(\F)$. 
\end{lema}
\begin{proof}
 Let $j:U\hookrightarrow X$ be the inclusion, then over $U$ we have $\mathcal{E}nd(\F)|_U \simeq \F|_U\otimes \F^\vee|_U$, so there is the evaluation map $\F|_U\otimes \F^\vee|_U\to \O_X|_U$, which divided by the generic rank of $\F$ is a section of the inclusion $\O_X\to \mathcal{E}nd(\F)$. As both $\F\simeq j_*(\F|_U)$ and $\O_X\simeq j_ *(\O_X|_U)$ by hypothesis, we have $\mathcal{E}nd(\F)\simeq j_*(\F|_U\otimes \F^\vee|_U)$ from which we get the morphism $\mathcal{E}nd(\F)\to\O_X$.   
\end{proof}

\begin{defi}
 A closed subset $Z\subseteq X$ is of \emph{relative codimension} $c$ over $S$ if and only if for every point $s\in S$ we have $\dim X_s - \dim Z_s=c$
\end{defi}

\begin{teo}\label{propPullBackD}
 Let $\F$ be a reflexive sheaf over $X$, flat over $S$, with $X$ smooth over $S$ and such that there is an open set $U\subseteq X$ with $\F|_U$ locally free and such that $Z:=X\setminus U$ is of relative codimension $\geq 2$.
 Let $f:R\to S$ be a morphism. Denote by $X_R$ the pull-back of $X$ by $f$ and by $f^*\F$ the corresponding pull-back of $\F$ as a sheaf over $X_R$ flat over $R$. Then the canonical morphism 
 \(
  f^\bullet D^{\leq 1}_X(\F)\to D^{\leq 1}_{X_R}(f^*\F)
 \) is an isomorphism of Lie algebroids.
\end{teo}
\begin{proof}
 Lemma \ref{lemafDD} gives the existence of a morphism $f^\bullet D^{\leq 1}_X(\F)\to D^{\leq 1}_{X_R}(\F)$. Here we will show the existence of an inverse to this morphism. 
 
 By Lemma \ref{lemaPartesPrincipales} there is a morphism 
 \[
  \mathcal{H}om(\mathcal{P}^1_{X_R}(f^*\F),f^*\F)\to f^*\left(\mathcal{H}om(\mathcal{P}^1_{X}(\F),\F)\right).
 \]
 Now, as $X\setminus U$ is of relative codimension $2$ over $S$ and $X\to S$ is smooth, then every section of $\O_X$ over $U$ extends to a global section, then we are in condition to apply Lemma \ref{lemaEnd} and get a splitting $\mathcal{E}nd(\F)\to\O_X$. The fact that $D^{\leq 1}_X(\F)$ is a pull-back of $\mathcal{H}om(\mathcal{P}^1_{X}(\F),\F)$ by $T_X$ over $T_X\otimes \mathcal{E}nd(\F)$ gives in turn a splitting $\mathcal{H}om(\mathcal{P}^1_{X}(\F),\F)\to D^{\leq 1}_X(\F)$. So we have a commutative diagram
 \[
  \xymatrix{
  \mathcal{H}om(\mathcal{P}^1_{X_R}(f^*\F),f^*\F) \ar[r]  & f^*\left(\mathcal{H}om(\mathcal{P}^1_{X}(\F),\F)\right) \ar@/^/[d] \\
  D^{\leq 1}_{X_R}(f^*\F) \ar[u] \ar[r] & f^*D^{\leq 1}_X(\F). 
  }
 \]
From the bottom arrow of this diagram and the pull-back property of $f^\bullet D^{\leq 1}_X(\F)$ we get a morphism
\[
 D^{\leq 1}_{X_R}(f^*\F) \to f^\bullet D^{\leq 1}_X(\F).
\]
Because $f^\bullet D^{\leq 1}_X(\F)$ is defined as a pull-back, the composition 
\[
f^\bullet D^{\leq 1}_X(\F)\to D^{\leq 1}_{X_R}(f^*\F) \to f^\bullet D^{\leq 1}_X(\F)
\]
is the identity, so $f^\bullet D^{\leq 1}_X(\F)\to D^{\leq 1}_{X_R}(f^*\F)$ is a monomorphism. The other composition, that is $D^{\leq 1}_{X_R}(f^*\F) \to f^\bullet D^{\leq 1}_X(\F)\to D^{\leq 1}_{X_R}(f^*\F)$ gives the central vertical morphism in the diagram
\[
 \xymatrix{
 0 \ar[r] & \mathcal{E}nd(f^*\F) \ar[r] \ar@{=}[d] & D^{\leq 1}_{X_R}(f^*\F) \ar[r] \ar [d] & T_{X_R} \ar@{=}[d] \\
 0 \ar[r] & \mathcal{E}nd(f^*\F) \ar[r]  & D^{\leq 1}_{X_R}(f^*\F) \ar[r]  & T_{X_R}, 
 }
\]
so $f^\bullet D^{\leq 1}_X(\F)\to D^{\leq 1}_{X_R}(f^*\F)$ is also an epimorphism.
\end{proof}

\begin{remark}
 If $a:\sA\to T_X$ is a Lie algebroid and $\alpha\in\sA(U)$ is a local section, then the map
 \[
 \psi|_U: \sA|_U\to\sA|_U
 \]
 defined by $ \psi|_U(\alpha):= \{\alpha,-\}$ is both a derivation of the Lie algebra structure of $\sA|_U$ and a differential operator of order $\leq 1$. 
Indeed,  we have that for each local section $f\in \O_X(U) $ and $\alpha \in \sA|_U$ we obtain   a section  $\psi f:=-a( \{\alpha,-\})( f)\in \O_X(U)$. Thus  by the   identity 
$$\{f \cdot \alpha, -\}=f\cdot\{\alpha, -\}- a( \{\alpha,-\})(f)\cdot \alpha. $$ 
  we conclude that 
 \[
(\psi|_U\circ m_f)(\alpha) -( m_f\circ \psi|_U)(\alpha)=m_{\psi f}(\alpha)  
  \]
 for all  $f\in \O_X(U) $ and $\alpha \in \sA|_U$. 
\end{remark}

We will denote by $\mathrm{Der}_{\mathrm{Lie}}(\sA)$ the sheaf of derivations of the Lie algebra structure of $\sA$. It is a sub-sheaf of the sheaf $\mathrm{End}_{\O_X}(\sA)$ of $\O_X$-linear endomorphisms.

\begin{defi}
Given a family of algebroids $(\sA_S,\{\cdot\,,\cdot\}_S)$, we denote by $\sigma:D^{\leq 1}_{X}(\sA)\to T_X$ the symbol map of differential operators, then we have  the diagram
 \[
  \xymatrix{
  \sA_S \ar[r] \ar[d]^{a_S} & \mathrm{Der}_{\mathrm{Lie}}(\sA_S)\cap D^{\leq 1}_{X} \ar[d]^\sigma \ar[r] & \left(\mathrm{Der}_{\mathrm{Lie}}(\sA_S)\cap D^{\leq 1}_{X}(\sA_S)\right)/\sA_S \ar[d]^{\sigma}\\
  T_{X|S} \ar[r] & T_X \ar[r] & \pi^* T_S.
  }
 \]
Where the intersection is taken as subsheaves of the sheaf $\mathcal{H}om_{f^{-1}\O_S}(\sA_S,\sA_S)$ of endomorphisms of $f^{-1}\O_S$-modules of $\sA_S$.
 We define the sheaf $\mathfrak{u}(\sA_S)$ as 
\[
 \mathfrak{u}(\sA_S):=\sigma^{-1}(\pi^{-1}T_S)\subseteq \left(\mathrm{Der}_{\mathrm{Lie}}(\sA_S)\cap D^{\leq 1}_{X}(\sA_S)\right)/\sA_S.
\]
Note that $\mathfrak{u}(\sA_S)$ is \emph{not} a coherent sheaf over $X$ but only a sheaf of $f^{-1}\O_S$-modules.
\end{defi}

\begin{remark}
The sheaf $\mathfrak{u}(\sA_S)$ inherits from $\mathrm{Der}_{\mathrm{Lie}}(\sA_S)$ the structure of a sheaf of Lie algebras. 
Indeed, as the inclusion $\sA_S \subseteq \mathrm{Der}_{\mathrm{Lie}}(\sA_S)$ is an ideal of Lie algebras, the Lie bracket $[\psi, \phi]=\psi\circ\phi-\phi\circ\psi$ passes to the quotient to a bracket in $\mathfrak{u}(\sA_S)$. 
\end{remark}

Let us begin by considering an unfolding $ (\sA,\{\cdot\,,\cdot\})$ of a family of algebroids  $ (\sA_S,\{\cdot\,,\cdot\}_S)$ parametrized by $S$.   
%
Note that, when the unfolding is transversal, we have that the anchor $a:\sA\to T_X$ determines an isomorphism $[a]: \sA/\sA_S\xrightarrow{\cong}\pi^* T_S$, so we can consider the morphism $\upsilon_{\sA}:\pi^* T_S\to \sA/\sA_S$ defined as the inverse of the isomorphism determined by the anchor.

\begin{prop}\label{upsilon}
If  $ (\sA,\{\cdot\,,\cdot\})$  is transversal to $S$ then $\upsilon_{\sA}(\pi^{-1}T_S)$ is a subsheaf of $\mathfrak{u}(\sA)$.
\end{prop}
\begin{proof}
Note that the statement is making reference to $\pi^{-1}T_S\subset \pi^{*}T_S$, that is the sheaf of vector fields that are constant along the fibers of $\pi$, also known as \emph{basic vector fields}.
Then,  given a local section $s\in \upsilon_{\sA}(\pi^{-1}T_S)\subseteq \sA/\sA_S$ we need to show that, for any lifting $\tilde{s}$ of $s$ in $\sA$ we have that $\{\sA_S, \tilde{s}\}\subseteq \sA_S$.
In other words we need to show that if $\alpha$ is a local section of $\sA_S$, then $a(\{\tilde{s},\alpha\})\in T_{X|S}$.
Locally in $X$ we can take $a(\tilde{s})$ of the form $Y+Z$ with $Y\in T_{X|S}$ and $Z\in  \pi^{-1}T_S$.

Noting $W=a(\alpha)\in T_{X|S}$ we compute
\[ 
a(\{\tilde{s},\alpha\})=[W, a(\tilde{s})]= [W, Y+Z] = [W,Y]- Z(W),
\] 
since  $W(Z)=0$, being $Z$ in $\pi^{-1}T_S$.
Then $a(\{\tilde{s},\alpha\})$ is in $T_{X|S}$, and also in $a(\sA)$, so it is in $a(\sA_S)$.
\end{proof}

\begin{teo}\label{correspondencia2}
Let $X$ be a non-singular variety, and $(\sA_S,\{\cdot\,,\cdot\}_S)$ a family of algebroids on $X$ parametrized by a scheme of finite type $S$.
 There is, for each scheme $S$, a $1$ to $1$ correspondence:
$$ \left\{   \begin{array}{c}
\text{transversal unfoldings } \\
\text{of $\sA_S$ } 
    \end{array} \right\} 
 \longleftrightarrow
\left\{   \begin{array}{c}
    \text{morphisms } \pi^{-1}(T_S)\to \mathfrak{u}(\sA_S)
\\ \text{ respecting brackets}
    \end{array} \right\}. $$

\end{teo}
\begin{proof}
Morphism associated to an unfolding: 
given a transversal unfolding $\sA$ of $\sA_S$ we have $\sA/\sA_S \cong \pi^*(T_S)$.
Then, it follows from  Proposition  \ref{upsilon} that we have  a map
 \[ \upsilon_{\sA}: \pi^{-1}(T_S)\to \mathfrak{u}(\sA_S).\]


Now, in order  to establish the first part of the correspondence we need to prove that this is a map of sheaves of $\pi^{-1}\O_S$-Lie algebras. 
Considere  $\upsilon_{\sA}(\pi^{-1}(T_S))\subseteq \mathfrak{u}(\sA_S)$ and we   take $A\subseteq \sA$ its pre-image under the morphism $\sA\to \sA/\sA_S$.
Since  $A/\sA_S$ is a subsheaf of $\mathfrak{u}(\sA_S)$, then $\sA_S\subseteq A$ is a Lie ideal, so we have a diagram
\[
 \xymatrix{
 \sA_S\ar[r] \ar[d]^a & A \ar[r] \ar[d]^a & \upsilon_{\sA}(\pi^{-1}(T_S) )\ar[d] \\
 a(\sA_S) \ar[r] & a(A) \ar[r] &\pi^{-1}(T_S),
 }
\]
as $\sA_S$ and $a(\sA_S)$ are Lie ideals of $A$ and $a(A)$ respectively, and the anchor $a$ is a morphism of Lie algebras. 
Then the morphism induced in the quotient $\upsilon_{\sA}(\pi^{-1}(T_S))\to  \pi^{-1}(T_S)$ is also a Lie algebra morphism, so also its inverse $\upsilon_{\sA}$  is a  morphism of sheaves of Lie algebras.
\\
\\
\noindent Unfolding associated with a morphism: given a morphism of sheaves of Lie algebras

\[
\upsilon: \pi^{-1}(T_S)\to \mathfrak{u}(\sA_S),
\]
we get an extension of sheaves of Lie algebras over $\pi^{-1}(\O_S)$

\[
 0\to \sA_S \to A \to  \pi^{-1}(T_S)\to 0.
\]

Indeed, the Lie algebra $A$ is defined as the pull-back of the diagram 
\[
 \xymatrix{
 A\ar[r] \ar[d] & \pi^{-1}(T_S)\ar[d] \\
 \mathrm{Der}_{\mathrm{Lie}}(\sA_S)\cap D^{\leq 1}_{X}(\sA_S) \ar[r] & \mathfrak{u}(\sA_S).
 }
\]
Moreover,  we have a morphism $\tilde{a}:A\to T_X$ defined by the composition $$A\to D^{\leq 1}_{X} \to T_X.$$
However, $A$ is not a quasi-coherent module over $X$, to get a module over $\O_X$ we need to modify this sheaf a little. 
For this we define the sheaf of sub-modules $B\subseteq A\otimes_{\pi^{-1}\O_S}\O_X$ as the quasi-coherent subsheaf generated by the stalks of the form $\alpha\otimes f - f \alpha\otimes 1 $, where $\alpha$ is a stalk of $\sA_S$ and $f$ a stalk of $\O_X$.
Then we define
\[
 \sA :=_{\mathrm{def}} (A\otimes_{\pi^{-1}\O_S} \O_X )/ B.
\]
The map $\tilde{a}$ can be extended to an $\O_X$-linear map $a':A\otimes \O_X \to T_X$. As $a'|_B=0$, we get an $\O_X$-linear map $\sA\to T_X$ extending the map $A\to T_X$.
Also notice that $\sA_S$ is a subsheaf of $\sA$ and that 
\[
 \sA/\sA_S= A/\sA_S\otimes \O_X=\pi^{-1}T_S\otimes \O_X=\pi^* T_S.
\]
We can extend the Lie bracket of $A$ to a Lie bracket in $A\otimes \O_X$ by the formula
\[
\{\alpha\otimes f, \beta\otimes g\}= \beta\otimes (f\cdot\tilde{a}(\alpha)(g))+\alpha\otimes(g\cdot\tilde{a}(\beta)(f))+\{\alpha,\beta\}\otimes f\cdot g.
\]
With this bracket the subsheaf $B$ is a sheaf of Lie ideals.
Therefore we get that $\sA$ has a Lie algebroid structure and it is an unfolding of $\sA_S$.

The construction of the morphism $ \pi^{-1}(T_S)\to \mathfrak{u}(\sA_S)$ associated with an unfolding and of the unfolding associated with the morphism are inverse to each other.
\end{proof}

\begin{prop}\label{propLieAlgebroid}
$\pi_* \mathfrak{u}(\sA_S)$ has the structure of a Lie algebroid over $S$.
\end{prop}
\begin{proof}
 Since  $\mathfrak{u}(\sA_S)$ is a $\pi^{-1}\O_S$-module then $\pi_* \mathfrak{u}(\sA_S)$ is an $\O_S$-module. 
 It is endowed with a Lie algebra bracket which is the push-forward of the bracket of $\mathfrak{u}(\sA_S)$. 
 Its anchor map can be defined as follows: 
 Taking the natural map $D^{\leq 1}_{X}\to T_X$ one gets by considering the derivation defined by a differential operator we get a diagram
 \[
  \xymatrix{
  \sA_S \ar[r] \ar[d]^{a_S} & \mathrm{Der}_{\mathrm{Lie}}(\sA_S)\cap D^{\leq 1}_{X} \ar[d]^\sigma \ar[r] & \mathfrak{u}(\sA_S) \ar[d]^{a_\mathfrak{u}}\\
  T_{X|S} \ar[r] & T_X \ar[r] & \pi^{-1} T_S.
  }
 \]
When the morphism $\pi$ is proper we have a natural isomorphism $\pi_* \pi^* T_S \cong T_S$.
 The anchor map  of $\pi_* \mathfrak{u}(\sA_S)$ is then 
 \[
  \pi_* a_{\mathfrak{u}}: \pi_* \mathfrak{u}(\sA_S)\to \pi_* \pi^{-1} T_S \cong T_S.
 \]
\end{proof}

Recall that a flat connection on an algebroid $\sA$ over a space $X$ is a section of the anchor map $s:T_X\to \sA$ respecting Lie brackets.
Then we get the following.
\begin{cor}
 There is a 1 to 1 correspondence 
$$ \left\{   \begin{array}{c}
    \text{transversal unfoldings }  \\
\text{of $\sA_S$ } 
    \end{array} \right\} 
 \longleftrightarrow
\left\{   \begin{array}{c}
    \text{flat connections on the algebroid }  \pi_*\mathfrak{u}(\sA_S)
    \end{array} \right\}. $$
\end{cor}

In particular to have an unfolding of a family $\sA_S$ of algebroids we must have an epimorphic anchor map on the algebroid $\pi_*\mathfrak{u}(\sA_S)$.
So for any family $\sA_S$ we have a foliation in the base space $S$ induced by the algebroid $\pi_*\mathfrak{u}(\sA_S)$. Any unfolding of a restriction of the family $\sA_S$ must be over a leaf of the considered foliation (compare with \cite{Genzmer}).

\begin{prop}
 Given a pull-back diagram of holomorphic spaces with smooth vertical arrows
 \[
  \xymatrix{
    Y \ar[r]^\phi \ar[d]_{\pi_R} & X \ar[d]^\pi \\
    R \ar[r]^f & S.
  }
 \]
And a family $\sA_S$ of algebroids over $\pi:X\to S$ such that $\mathrm{sing}(\sA)$ has relative codimension greater than $2$ as a subscheme of $X/S$. We have a canonical morphism
\[
 \pi_{R*}\mathfrak{u}(\phi^\bullet \sA)\to f^\bullet \pi_{*}\mathfrak{u}(\sA),
\]
as algebroids on $R$. 
\end{prop}
\begin{proof}
 By remark \ref{remarkpullback} we have that the sheaf underlying the algebroid $\phi^\bullet\sA$ is $\phi^*\sA$. 
 
 By the hypothesis on $\mathrm{sing}\sA$ we can apply Proposition \ref{propPullBackD} which says that we have an isomorphism $\phi^\bullet D^{\leq 1}_X(\sA)\simeq D^{\leq 1}_{Y}(\phi^*\sA)$.
 Also, as the Lie algebra structure of $\sA$ is $\O_S$-linear, and $\O_Y= \O_X\otimes_{\O_S}\O_R$, then the Lie algebra structure of $\phi^*\sA$ is $\O_R$-linear.
 If a local section $\varphi\in D^{\leq 1}_Y(\phi^*\sA)(U)$ acts on $\phi^*\sA(U)$ as a derivation of the Lie algebra structure, then it acts as a derivation on sections of $\phi^*\sA$ of the form $s\otimes 1$, so the image of $\varphi$ by the composition 
 $$ D^{\leq 1}_{Y}(\phi^*\sA) \simeq \phi^\bullet D^{\leq 1}_X(\sA) \to \phi^* D^{\leq 1}_X(\sA)$$
 is in the subsheaf $\phi^{*}(D^{\leq 1}_X(\sA)\cap \mathrm{Der}_{\mathrm{Lie}}(\sA))$. The fact that the diagram
 \[
  \xymatrix{
  \phi^\bullet D^{\leq 1}_X(\sA)  \ar[r] \ar[d] & \phi^* D^{\leq 1}_X(\sA) \ar[d]\\
  T_Y \ar[r] & \phi^*T_X
  }
 \]
 commutes implies with the above that we have a morphism $\mathfrak{u}(\phi^\bullet\sA)\to \phi^*\mathfrak{u}(\sA)$ which in turn gives a morphism $\mathfrak{u}(\phi^\bullet\sA)\to \phi^\bullet\mathfrak{u}(\sA)$. The Proposition follows from the fact that $\pi_{R*}\phi^\bullet\mathfrak{u}(\sA)\simeq f^\bullet\pi_*\mathfrak{u}(\sA) $
\end{proof}

\section{Examples}
 \label{sec4}

\subsection{Holomorphic foliations}

Let $a_S\colon \F_S \to\T_{X|S}$
a  family of  singular holomorphic   foliations  over $X$ . 
An unfolding of $\F_S$  is a  foliation  $\F$ on $ X$ with anchor $a:\F \to T_X$ such that 
that
\begin{itemize}
	\item[(a)]    $\F_S= a^{-1}(a_S(\F_S) )$. 
	\item[(b)]  $\dim( \F) =  \dim( \F_S)  + \dim(S).$
\end{itemize}

Now suppose we have $\F=\mathcal{L}$ is a line bundle, in other words we have a foliation by curves. 
Let $v$ be a local generating section of $\mathcal{L}$  and $\psi$ a local $\CC$-linear endomorphism of $\mathcal{L}$, so $\psi(v)=f_\psi \cdot v$ for some local section $f_\psi$ of $\O_X$.
If $\psi$ is a derivation for the Lie algebra structure of $\mathcal{L}$ (which is induced by the inclusion $\mathcal{L}\subseteq T_X$) Then for any local section $g$ of $\O_X$ we get
\begin{align*}
	\psi([v,g\cdot v]) &= [\psi(v), g\cdot v]+[v, \psi(g\cdot v)],\\
	\psi(v(g)\cdot v) &= [f_\psi\cdot v, g\cdot v]+[v,\psi(g\cdot v)]
\end{align*}	
If $\psi$ is also a differential operator we have $\psi(v(g)\cdot v)-v(g)\psi(v)= \sigma(\psi)(v(g))\cdot v$, where $\sigma$ denotes the symbol of $\psi$. Then we have

\begin{align*}
	\psi(v(g)\cdot v) &= [f_\psi\cdot v, g\cdot v] + [v,\psi(g\cdot v)]=\\
v(g)f_\psi\cdot v - \sigma(\psi)(v(g))\cdot v &= f_\psi v(g)\cdot  v + v(g)f_\psi\cdot v - v(\sigma(\psi)(g))\cdot v.
\end{align*}

In other words, $[\sigma(\psi), v](g) = f_\psi v(g)$, as this happens for every local section $g$ of $\O_X$ then 
\begin{equation}\label{equnfol}
	[\sigma(\psi), v] = f_\psi \cdot v = \psi(v).
\end{equation}

Denoting $p:T_X\to \pi^*T_S$ the projection, lets call for an open set $V\subseteq X$
\[
U(\mathcal{L})(V):=\left\{\psi \in \left(\mathrm{Der}_{\mathrm{Lie}}(\mathcal{L})\cap D^{\leq 1}_{X}(\mathcal{L})\right)(V),\ \  \text{s.t: } p\circ \sigma(\psi)\in \pi^{-1}T_S\right\}.
\]
We have then $\mathfrak{u}(\mathcal{L})=U(\mathcal{L})/\mathcal{L}$ and an inclusion of short exact sequences
\begin{equation}\label{diagramunfol}
	\xymatrix{
	0 \ar[r] & K \ar[r] \ar@{^{(}->}[d] & U(\mathcal{L})	\ar[r] \ar@{^{(}->}[d] & a(U(\mathcal{L})) \ar[r]  \ar@{^{(}->}[d]  & 0 \\
	0 \ar[r] & \O_X \ar[r] & D^{\leq 1}_{X}(\mathcal{L}) \ar[r]^\sigma  & T_X \ar[r] & 0.
}
\end{equation}

Notice that equation (\ref{equnfol}) implies that sections of $a(U(\mathcal{L}))$ act on $\mathcal{L}$ as differential operator, so the top short exact sequence in diagram (\ref{diagramunfol}) splits, so every section  $\psi$ of $U(\mathcal{L})$ can be written as $\psi = m_a + Y$ where $m_a$ is multiplication by a local section $a$ of $\O_X$ and $Y$ is a local vector field on $X$. Moreover, equation (\ref{equnfol}) implies that if $m_a + Y$ is a section of $U(\mathcal{L})$ then for a section $X$ of $\mathcal{L}$ we have $[Y,X]=[Y,X]+a\cdot X$, so $a=0$ then the sheaf $K$ of diagram (\ref{diagramunfol}) is null.
In conclusion we can characterize $\mathfrak{u}(\mathcal{L})$ as
\[
	\mathfrak{u}(\mathcal{L})= \left( Y \in T_X : p(Y) \in \pi^{-1}T_S,\ [Y,\mathcal{L}]\subseteq \mathcal{L} \right)/\mathcal{L}.
\] 
In this case the kernel of the algebroid $\pi_*\mathfrak{u}$ is the $\O_S$-linear Lie algebra 
\[
	\mathfrak{g}(\mathcal{L})=\left( Y\in T_{X|S} : [Y, \mathcal{L}]\subseteq \mathcal{L} \right)/\mathcal{L},
\] 
which is the algebra of infinitesimal symmetries of the foliation.

\subsection{Sheaf   of Lie algebra}
Let   $  \sA_S  $ be a family of sheaves of Lie algebras. In this case the anchor map $a_S=0$.
An unfolding of $\sA_S$  is a Lie algebroid $\sA$ on $ X$ with anchor $a:\sA\to T_X$ such 
that
\begin{itemize}
\item[(a)]  The family of Lie algebroids $\sA_S$, is recovered as $\sA_S= a^{-1}(0 )=Ker(a)$. 
\item[(b)]  $\mathrm{rank}(\sA) =  \mathrm{rank}(\sA_S)  + \dim(S).$
\end{itemize}
That is, $  \sA_S  $  is an isotropy sub-Lie algebroid of a  Lie algebroid  $a:\sA\to T_X$ and the dimension of the foliation associated this Lie algebroid
has dimension equal to  $\dim(S)$ by the condition $b)$.
We have 
$$
0\to   \sA_S \to \sA\to Im(a) \to 0.
$$
Therefore, if  the unfolding is transversal, we have  the  isomorphism $[a]: \sA/\sA_S\xrightarrow{\cong}\pi^* T_S$. That is $ Im(a) \cong \pi^* T_S$, this in turn define a splitting of the short exact sequence $0\to T_{X|S}\to T_X\to \pi^*T_S\to 0$. 

In particular we can take any sheaf $\F$ flat over $S$ and take the algebroid $\sA_S$ to be $\F$ with the structure of an abelian Lie algebra and the zero anchor map. 
In this case we get the extension of Lie algebras 
\[
 0\to \F\to \sA \to \pi^*T_S\to 0.
\]
This extension defines an action of $\pi^*T_S$ on $\F$, in particular we have a flat connection on $\pi_*\F$. The extension also defines an homology class $c$ on the Chevalley-Eilenberg cohomology $c\in H^2(\pi^*T_S,\F)$.
Reciprocally, given a splitting of $0\to T_{X|S}\to T_X\to \pi^*T_S\to 0$, and a flat connection $\nabla$ on the quasi-coherent sheaf $\pi_*\F$, we get a Lie algebra action of $\pi^*T_S$ on $\F$.
Indeed, let $p\in X$, $v$ be a local section of ${\pi^*T_S}_p$ and $x\in \F_p$. As $\F_p$ is a localization of $\F_{\pi^{-1}(\pi(p))}$ we can write $x$ as $\sum_i f_i y_i$ with $y_i\in \F_{\pi^{-1}(\pi(p))}$ and $f_i\in \O_{X,p}$, we can also assume $v=g\cdot w$ with $w\in T_{S,\pi(p)}$ and $g\in \O_{X,p}$. Now, denoting by $\iota:\pi^*T_S\to T_X$ the splitting, we can define the action of $\pi^*T_S$ in $\F$ as
\[
 \nabla_{v}(x)= \sum_i g\cdot\iota(w)(f_i) y_i + f_i \nabla_w(y_i). 
\]

Now,  given an element of the Chevalley-Eilenberg cohomology $c\in H^2(\pi^*T_S,\F)$, where $\F$ is taken as a $\pi^*T_S$-module with the action just defined, we get an abelian extension of Lie algebras
\[
 0\to \F \to \sA \to \pi^*T_S\to 0.
\]

Which is an unfolding of $\F$ as abelian Lie algebroid with trivial anchor.

\subsection{Poisson structures}
Let  $a_S:(\Omega_{X|S}^1,\{\cdot\,,\cdot\}_S)\to T_{X|S}$ be a family of holomorphic Poisson structures over a smooth morphism $\pi:X\to S$.  
A Poisson structure $a:(\Omega_{X}^1,\{\cdot\,,\cdot\})\to T_{X}$ on $ X$   is   an unfolding of 
$a_S:(\Omega_{X|S}^1,\{\cdot\,,\cdot\}_S)\to T_{X|S}$ if $\Omega_{X|S}^1$ is the  pre-image by $a$ of the  associated  symplectic foliation of $(\Omega_{X|S}^1,\{\cdot\,,\cdot\}_S)$, since $\mathrm{rank}(\Omega_{X}^1) =  \mathrm{rank}(\Omega_{X|S}^1)  + \dim(S).$


We have a diagram

\begin{equation} 
	\xymatrix{
	0 \ar[r] &  \pi^*\Omega_S^1 \ar[r]^{i} \ar[d]^{\rho} &   \Omega_{X}^1	\ar[r] \ar[d]^{a} &  \Omega_{X|S}^1 \ar[r]  \ar[d]^{a_S}  & 0 \\
   &  \pi^*T_S & T_X \ar[l]^{i^*}  & \ar[l]  T_{X|S}  & \ar[l] 0 ,
}
\end{equation}
where $\rho:=i^* \circ a\circ i$.
This implies that the map $\rho: \pi^*\Omega_S^1 \to  \pi^*T_S$   induces a Poisson structure on $S$ by  $ \pi_*\rho: \Omega_S^1 \to  T_S$.

If  the unfolding is transversal, we have that the  isomorphism $[a]: \Omega_{X}^1/\Omega_{X|S}^1 \to  \pi^*T_S$ provides   a splitting for the sequence   
\begin{equation} 
	\xymatrix{
	0 \ar[r] &  \pi^*\Omega_S^1 \ar[r]    &   \Omega_{X}^1	\ar[r]  &  \Omega_{X|S}^1  \ar[r] & 0
}
\end{equation}
 which implies that $\rho: \pi^*\Omega_S^1 \to  \pi^*T_S$ is an isomorphism, i.e, $ \pi_*\rho: \Omega_S^1 \to  T_S$ is a symplectic structure on $S$.  


It would be interesting to study how  unfoldings of holomorphic Poisson structures  behave under Morita equivalence \cite{Correa}.

\subsection{Sheaf of logarithmic forms}

Let $X$ be a smooth projective variety and $D$ an effective normal
crossing divisor on $X$.  Denote $\imath: T_{X}(-\log D)\to T_{X}$ the inclusion anchor map. 
A deformation of the pair $(X,D)\to S$ can   be interpreted as  a family of Lie algebroids $\imath_S: T_{X|S}(-\log D) \to   T_{X|S} $, where $T_{X|S}(-\log D) = \imath(T_{X}(-\log D))\cap  T_{X|S}$ and $\imath_S:=\imath|_S$.   Since    the rank of $T_{X|S}(-\log D) $ is  $\dim(X) - \dim(S)$ and $\imath^{-1}(T_{X|S}(-\log D))=T_{X}(-\log D)$,  then  $\imath: T_{X}(-\log D)\to T_{X}$ is an unfolding of $T_{X|S}(-\log D)$.

If  the unfolding is transversal, we have  the  isomorphism $$[\imath]:  T_{X}(-\log D)/ T_{X|S}(-\log D) \xrightarrow{\cong}\pi^* T_S.$$
We have the  holomorphic Bott's partial connection on $T_{X}(-\log D)/ T_{X|S}(-\log D)$
\[
\nabla: T_{X}(-\log D)/ T_{X|S}(-\log D) \to \Omega^1_{X|S}(\log D) \otimes [T_{X}(-\log D)/ T_{X|S}(-\log D)]
\]
by setting
\[
\nabla_u(q)=\phi([i_S(u),\tilde q]),
\]
where $\phi: T_{X}(-\log D)\to T_{X}(-\log D)/ T_{X|S}(-\log D)$ denotes the projection,  $\tilde q \in  T_{X}(-\log D)$ such that
$\phi(\tilde q)=q$ and $u \in  T_{X|S}(-\log D)$.  Since $\nabla$ is flat along $ T_{X|S}(-\log D)$ and  the  unfolding is transversal we conclude that it induces a holomorphic connection on $T_S$ given by $\tilde{\nabla}:= \pi_*(\nabla \circ  [\imath])$.

\subsection*{Acknowledgments}
The authors  thank the anonymous referees for giving many suggestions that helped improving the
presentation of the paper.

{\tiny

{

}
\

\

\

{\small
\noindent
\begin{tabular}{l l}
Mauricio Corr\^ea$^*$ \hspace{2cm}\null&\textsf{mauriciojr@ufmg.br,mauricio.barros@uniba.it}\\
Ariel Molinuevo$^\dag$  &\textsf{amoli@im.ufrj.br}\\
Federico Quallbrunn$^\ddag$  &\textsf{fquallb@dm.uba.ar}\\
\end{tabular}}

\

\

\

{\tiny
\noindent
\begin{tabular}{l l l}
$^*${ICEx - UFMG} & $^\dag$Instituto de Matem\'atica& $^\ddag$ {Departamento de Matem\'atica}  \\
{Departamento de Matem\'atica} &Av. Athos da Silveira Ramos 149 & {Universidad CAECE} \\
{Av. Antonio Carlos 6627} &Bloco C, Centro de Tecnologia, UFRJ  & {Av. de Mayo 866}\\
{CEP 30123-970} &Cidade Universit\'aria, Ilha do Fund\~ao & {CP C1084AAQ} \\ 
{Belo Horizonte, MG}&CEP 21941-909  &  {Ciudad de Buenos Aires}  \\
{Brasil}  &Rio de Janeiro, RJ &{Argentina}   \\
&Brasil &
\\
$^*$Universit\`a degli Studi di Bari Aldo Moro & &
\\
Dipartimento di Matematica&&
\\
Via E.Orabona 4, 70125 Bari& &
\\
Italy& &
\end{tabular}}

\end{document}